\UseRawInputEncoding
\documentclass[12pt]{article}
\usepackage{amssymb,amsmath}
\usepackage{latexsym,bm}
\usepackage{dsfont}

\usepackage{graphicx,amssymb,mathrsfs,amsmath,latexsym,amsfonts,amsthm}
\usepackage{amsmath,amssymb,amsthm}
\usepackage{latexsym}
\usepackage{amssymb}
\usepackage{stmaryrd}
 \usepackage{float}
\usepackage{psfrag}
\usepackage{xcolor}
\usepackage{enumerate}
\usepackage{manfnt,mathrsfs}
\usepackage{epstopdf}

\makeatletter 
\@addtoreset{equation}{section}
\makeatother
\newtheorem{theorem}{Theorem}

\newtheorem{corollary}[theorem]{Corollary}

\newtheorem{lemma}[theorem]{Lemma}

\newtheorem{problem}[theorem]{Problem}
\def\adots{\mathinner{\mkern2mu\raise0pt\hbox{.}  
\mkern2mu\raise4pt\hbox{.}\mkern1mu
\raise7pt\vbox{\kern7pt\hbox{.}}\mkern1mu}}

\def\tu{{\rm Tur{\'a}n\,}}
\date{ }
\begin{document}
\title{ Extremal digraphs avoiding distinct walks of length 3 with the same endpoints}
\author{ Zejun Huang\thanks{College of Mathmatics and Statistics, Shenzhen University, Shenzhen,  518060, P.R. China.  (mathzejun@gmail.com) }, ~~Zhenhua Lyu\thanks{School of Science, Shenyang Aerospace University, Shenyang,  110136, P.R. China. (lyuzhh@outlook.com)}  } \maketitle

\begin{abstract}
In this paper, we determine the maximum size of digraphs on $n$ vertices   in which there are no  two distinct walks of length $3$ with the same initial vertex and the same terminal vertex. The  digraphs attaining this maximum size are also characterized. Combining this with previous results, we obtain a full solution to a problem  proposed by X. Zhan in 2007 .
\end{abstract}

{\bf Key words:}
digraph, Tur\'an problem,  walk

{\bf 2010 Mathematics Subject Classifications:} 05C20, 05C35
\section{Introduction and main results}
The Tur\'an-type extremal problem is an important topics in extremal graph theory, which concerns the maximum number of edges in graphs containing no given subgraphs and the extremal graphs achieving this maximum. Most of the previous results on Tur\'an problems concern
undirected graphs and only a few Tur\'an problems on digraphs have been investigated; see \cite{BES,BH,BS,HL,HL2,MRT,AS} and the references therein.

A {\it simple digraph} is a digraph that does not contain multiple arcs but allows loops. A {\it strict digraph} is a  loopless simple digraph.
Digraphs in this paper are simple unless otherwise stated.  We follow the terminology and notations  in \cite{bm}. Directed
walks, directed paths and directed cycles are abbreviate as walks, paths and cycles, respectively. The number of vertices in a digraph is called its
{\it order} and the number of arcs its {\it size}. Given a digraph $D$ and a set of digraphs $F$, $D$ is said to be {\it $F$-free} if $D$ contains no  digraph in $F$ as its subgraph.

	Denote by $\overrightarrow{K}_n$ the complete   digraph on $n$ vertices. A natural \tu problem on digraphs is determining the maximum size of   $\overrightarrow{K}_r$-free digraphs of a given order, which has been solved in \cite{JM}. Brown and Harary \cite{BH}  determined the precise extremal size of digraphs as well as the extremal digraphs that avoids a tournament. They also studied digraphs avoiding a direct sum of two tournaments, or a digraph on at most 4 vertices where any two vertices are joined by at least one arc.  In \cite{HL2}, we determined the  extremal size of digraphs avoiding an orientation of the 4-cycle as well as the extremal digraphs. By adopting dense matrices, Brown, Erd\H os, and Simonovits \cite{BES,BES3} presented asymptotic results on  extremal digraphs avoiding a family of digraphs. Howalla, Dabboucy and Tout \cite{HDT1,HDT2} determined the extremal sizes of  the $\overrightarrow{K}_2$-free digraphs avoiding $k$  paths with the same initial vertex and terminal vertex for $k=2,3$. Maurer, Rabinovitch and Trotter \cite{MRT} studied the extremal  $\overrightarrow{K}_2$-free transitive digraphs which  contain at most one  path from $x$ to $y$  for any two distinct vertices $x,y.$

Denote by
$F_k$ the family of  digraphs consisting of two different walks of length $k$ with the same initial vertex and the same terminal vertex.  Let $ex(n,F_k)$ be  the maximum size of $F_k$-free digraphs of order $n$, and let $EX(n,F_k)$ be the set of $F_k$-free digraphs of order $n$ with size $ex(n,F_k)$. We are interested in the following \tu-type problem.
\begin{problem}\label{p1}
Given positive integers $n$ and $k$, determine $ex(n,F_k)$ and  $EX(n,F_k)$.
\end{problem}

Let $D=(V, {A})$ be a digraph with vertex set $ {V}=\{v_1,v_2,\ldots,v_n\}$ and arc set $ {A}$. Its {\it adjacency matrix} $A_D=(a_{ij})$ is defined by
\begin{equation}\label{eqhh1}
a_{ij}=\left\{\begin{array}{ll}
1,&\textrm{if } (v_i,v_j)\in  {A};\\
0,&\textrm{otherwise}.\end{array}\right.
\end{equation}
Conversely, given an   0-1 matrix $A=(a_{ij})_{n\times n}$, we can define its digraph $D(A)=( {V}, {A})$ on vertices $v_1,v_2,\ldots,v_n$ by (\ref{eqhh1}), whose adjacency matrix is $A$. Considering the entries of $A^k(D)$, Problem \ref{p1} is equivalent to the following matrix problem, which was proposed by X. Zhan in 2007; see \cite[p.234]{MT}.

\vskip 0.3cm
{\bf Problem 1'.} {\it Determine the maximum number of ones in a 0-1 matrix $A$ of order $n$ such that $A^k$ is also a 0-1 matrix. Characterize the extremal 0-1 matrices attaining this maximum.}
	\vskip 0.3cm

In 2010, Wu \cite{W} solved Problem \ref{p1} for the case $k=2$.  Huang and Zhan \cite{HZ} solved the case $k\ge n-1\ge 4$ and determined  $ex(k+2,F_{k})$, $ex(k+3,F_{k})$ when $k\ge 4$. Huang, Lyu and Qiao \cite{HLQ} determined $ex(n,F_{k})$ for $4\le k\le n-4$, and  characterized $EX(n,F_{k})$ for $5\le k\le n-4$. They also characterize $EX(k+2,F_{k})$ and $EX(k+3,F_{k})$ when $k\ge 4$. Lyu \cite{lyu}  characterized $EX(n,F_{4})$ when $n\ge 8$. In this paper, we solve the last remained case $k=3$ when $n\ge 16$.\\

To state our results, we need the following notations and definitions. For a digraph $D=(V,A)$,  we denote by $a(D)$   the  size of $D$.
For $i,j\in V$, if $D$ contains an arc from $i$ to $j$, we say that $j$ is a {\it successor} of $i$, and $i$ is a  {\it predecessor} of $j$, denoted $ij$ or $i\rightarrow j$. A directed walk $xv_1v_2\cdots v_ky$ is called a {\it $xy$-walk}, where $x$ is its {\it initial vertex} and $y$ is its {\it terminal vertex}.
   For $S,T\subset V$, we denote by $D[S]$   the subgraph of $D$ induced by $S$,  $A(S,T)$ the set of arcs from $S$ to $T$, $a(S,T)$  the cardinality of $A(S,T)$.
Let
$$N^+(u)=\{x\in V|ux\in A\} \quad \text{and}\quad N^-(u)=\{x\in V|xu\in A\}$$
be the out-neighbour set and in-neighbour set a vertex $u$, respectively. The {\it out-degree} and {\it in-degree} of  $u$ are $d^+(u)\equiv |N^+(u)|$ and $d^-(u)\equiv |N^-(u)|$, respectively. Let $\Delta^+(D)$ denote the maximum out-degree of $D$, which are  abbreviate as $\Delta^+$ if there is no confusion. Given $X\subseteq V$, $d^+_{X}(u)$ and $d^-_{X}(u)$ are the numbers of the successors and predecessors of $u$ from $X$, respectively.

Two digraphs $D_1=(V_1,A_1)$ and $D_2=(V_2,A_2)$ are {\it  isomorphic}, written $D_1\cong D_2$, if there is a bijection $\sigma: V_1\rightarrow V_2$ such that $(u,v)\in A_1$ if and only if $(\sigma(u),\sigma(v))\in A_2$.

For a digraph $D=(V,A)$ with
$V=\{v_1,v_2,\ldots,v_n\}$,
a {\it blow-up} of $D$ is obtained by replacing every
vertex $v_i$ with a finite collection of copies of $v_i$, denoted  $V_i$, so that $xy$ is an arc
for $x\in V_i$ and $y\in V_j$ if and only if $v_iv_j\in A(D)$. A blow-up of a digraph is said to be {\it balanced} if $|V_i|$ and $|V_j|$ differ by at most one  for any pair $i\ne j$. Denote by  $B(V_1,V_2)$ the blow-up of an arc with vertex set $V_1\cup V_2$ such that $xy$ is an arc in $B(V_1,V_2)$ if and only if $x\in V_1, y\in V_2$, and $T(V_1,V_2,V_3)$   the blow-up of the transitive tournament of order 3 such that $xy$ is an arc in $T(V_1,V_2,V_3)$ if and only if $x\in V_i, y\in V_j$ with $i<j$.

Given a digraph $D$, denote by $D+e$ the digraph obtained from $D$ by adding an arc $e$, where the arc $e$ is allowed to be a loop.
 We define $H(V_1,V_2)$ to be the  digraph  obtained from $B(V_1,V_2)$ by adding an arc or a loop with both ends from $V_1$. Let
\begin{eqnarray*}
T_{3,n}=&\{&T(V_1,V_2,V_3)+e: |V_1|+|V_2|+|V_3|=n, |V_i|=\left\lfloor \frac{n}{3}\right\rfloor ~{\rm or}~ \left\lceil \frac{n}{3}\right\rceil ~{\rm for}~ i=1,2,3, \\
 &&  e \text{ is an arc with both ends from }  V_2  \},
\end{eqnarray*}
 which is the set of digraphs   obtained from a balanced blow-up of the transitive tournament of order 3 by embedding an arc in the middle partite set.
 Then each digraph in $T_{3,n}$ has one of the diagrams in Figure \ref{f1}.
\begin{figure}[H]\label{f1}
		\centering
		\includegraphics[width=2.7in]{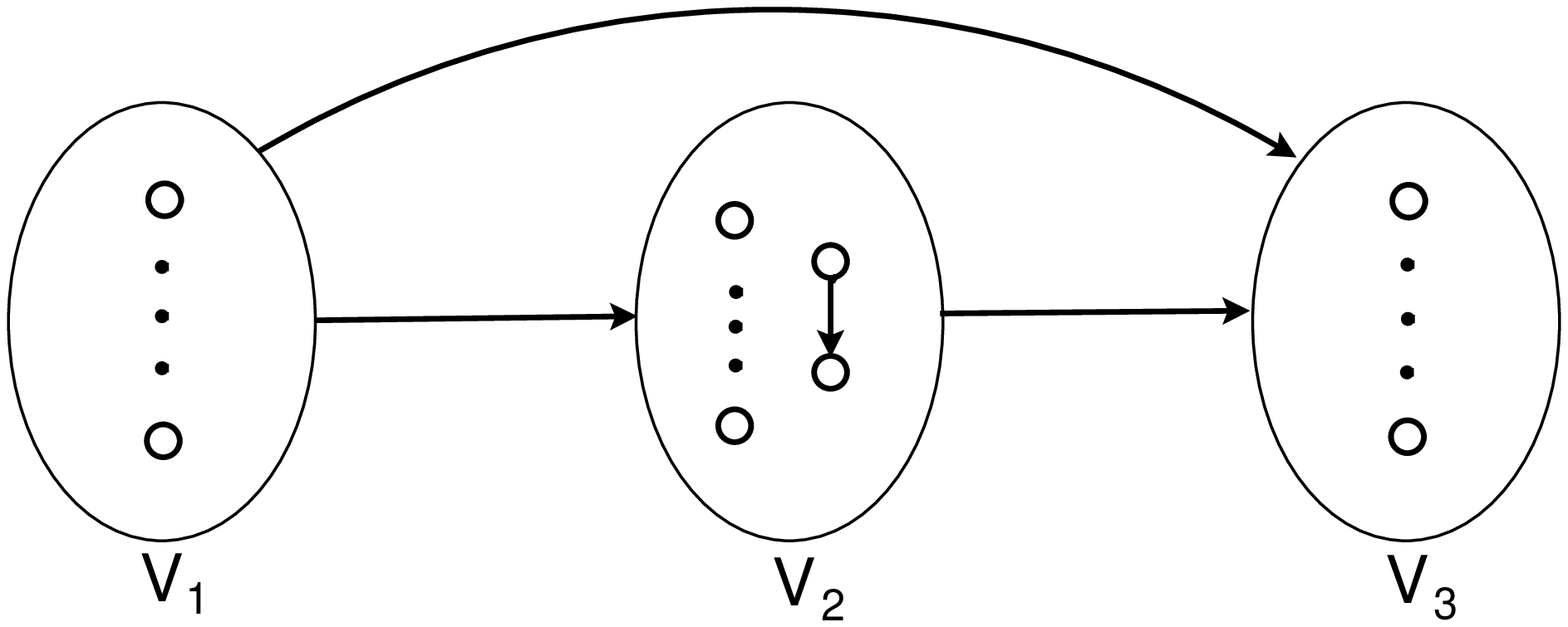}\hspace{0.8cm}
\includegraphics[width=2.7in]{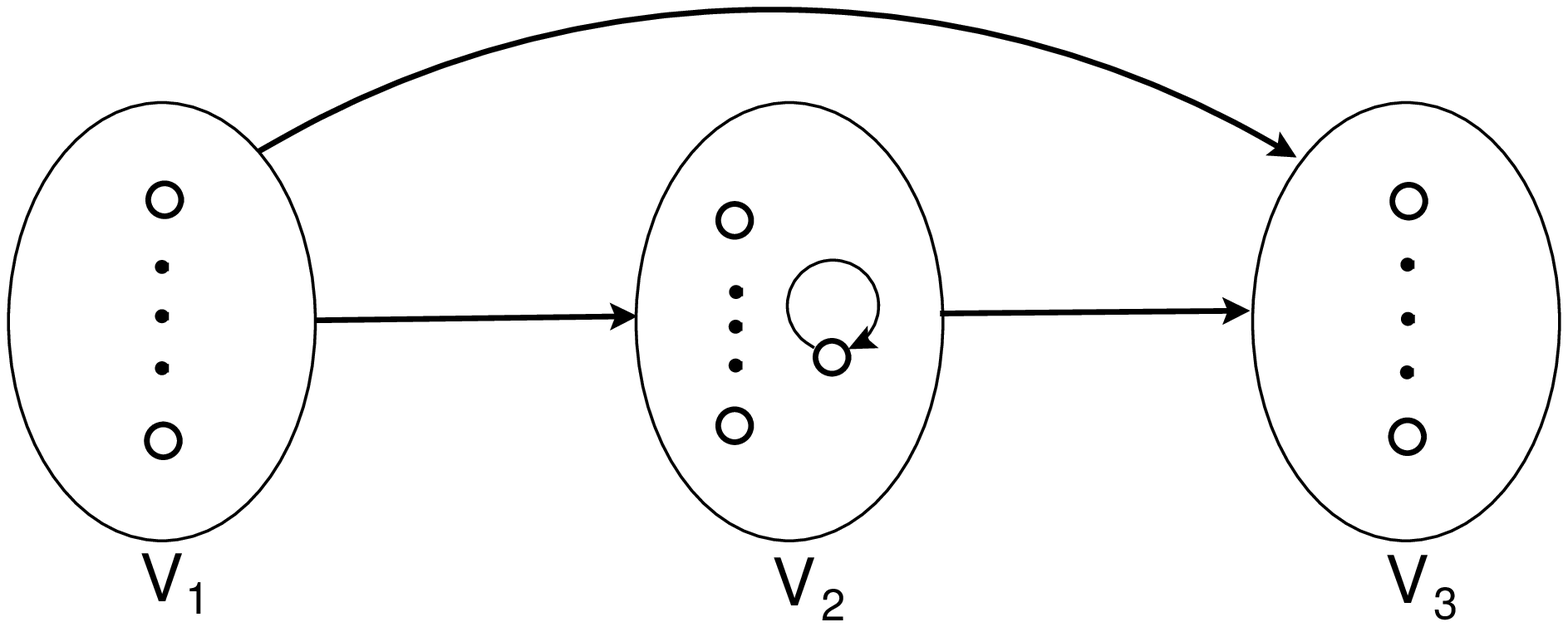}\\
\caption{The diagrams of the digraphs in $T_{3,n}$.}
\end{figure}

Now we state our main result as follows.

\begin{theorem}\label{th1}
  Let $n\ge 16$ be an integer.   Then
\begin{equation*}\label{eq15}
ex(n,F_3)= \left\lfloor \frac{n^2}{3}\right\rfloor+1 \quad and\quad EX(n,F_3)=T_{3,n}.
\end{equation*}
 \end{theorem}

Notice that $T_{3,n}$ contains some loopless digraphs. By Theorem \ref{th1} we have the following result for strict digraphs.
 \begin{corollary}
 Let $D$ be an $F_3$-free strict digraphs on $n\ge 16$ vertices. Then
 $$a(D)\le \left\lfloor \frac{n^2}{3}\right\rfloor+1$$
with equality if and only if $D\in T_{3,n}$.
 \end{corollary}

{\it Remark.}  Combining Theorem \ref{th1}  with previous results, we now obtain a full solution to Problem 1 (Problem 1').

\section{Proofs}
We will need the following lemmas to prove Theorem \ref{th1}.

\begin{lemma}\label{le3}
Let $D$ be a digraph on $n\ge 8$ vertices such that there are no  two walks of length 2 with  the same terminal vertex. Then
\begin{equation}\label{eq7}
a(D)\le \left\lfloor\frac{n^2}{4}\right\rfloor+1
\end{equation}with equality if and only if $D$ is isomorphic to $H(V_1,V_2)$ with $\{|V_1|,|V_2|\}=\{\lfloor n/2\rfloor,\lceil n/2\rceil\}$.
\end{lemma}
\begin{proof}
Suppose $D=(V,A)$. Let $V_1=\{u\in V|d^+(u)>0\}$ and let $V_2=V\setminus V_1$. Then we have
\begin{equation}\label{eq11}
a(V_2,V)=0
\end{equation} and
\begin{equation*}
d^-(u)\le 1 \quad {\rm for~ all}\quad u\in V_1.
\end{equation*}
Let $V'_1=\{u\in V_1|d^-(u)=0\}$ and   $V''_1=\{u\in V_1|d^-(u)=1\}$. Then $V_1=V'_1\cup V''_1$ and $V'_1\cap V''_1=\emptyset$. Moreover,
\begin{equation}\label{eq12}
a(V_1,V_1)=\sum\limits_{u\in V_1}d^-(u)=|V''_1|.
\end{equation}

If $V''_1$ is empty, we have
\begin{eqnarray}\label{eqh1}
a(D)&=&a(V_1,V_2)+a(V_1,V_1)+a(V_2,V)\nonumber\\
&=&a(V_1,V_2)+|V''_1|+0\nonumber\\
&\le& |V_1||V_2|
\le \left\lfloor\frac{n^2}{4}\right\rfloor,
\end{eqnarray}
which implies (\ref{eq7}).

Now suppose $|V''_1|\ge 1$. Then we have
\begin{equation*}\label{eq14}
a(V''_1,V_2)\le |V_2|.
\end{equation*} Otherwise $a(V''_1,V_2)>|V_2|$ implies that there exist $u_1,u_2\in V''_1$ sharing a common successor in $V_2$, which contradicts the given condition. Notice that $$|V''_1||V_2|+1\ge  |V_2|+|V''_1|.$$
We have
\begin{eqnarray}
a(D)&=&a(V_1,V_2)+a(V_1,V_1)+a(V_2,V)\nonumber\\
&=&a(V_1',V_2)+a(V''_1,V_2)+|V''_1|\nonumber\\
&\le& |V'_1||V_2|+|V_2|+|V''_1|\nonumber\\
&=&|V_1||V_2|-|V''_1||V_2|+|V_2|+|V''_1|\nonumber\\
&\le& |V_1||V_2|+1\nonumber\\
&\le& \left\lfloor\frac{n^2}{4}\right\rfloor+1\label{eqh22}.
\end{eqnarray}
Hence, (\ref{eq7}) holds.

If equality in (\ref{eq7}) holds,  then $|V''_1|\ge 1$, and all inequalities in (\ref{eqh22}) are equalities. We can deduce that $||V_1|-|V_2||\le 1$, $a(V_1',V_2)=|V_1'||V_2|$, $a(V''_1,V_2)=|V_2|$ and $|V''_1|=1$. Therefore,  $D$ is isomorphic to $H(V_1,V_2)$ with $\{|V_1|,|V_2|\}=\{\lfloor n/2\rfloor,\lceil n/2\rceil\}$.
\end{proof}

\begin{corollary}\label{cor1}
Let $D$ be a digraph on $n\ge 8$ vertices without two walks of length 2 sharing the same terminal vertex such that
\begin{equation}\label{eq13}
a(D)= \left\lfloor \frac{n^2}{4} \right\rfloor.
\end{equation}  Then $D$ is isomorphic to a digraph obtained from $H(V_1,V_2)$ with $\{|V_1|,|V_2|\}=\{\lfloor n/2\rfloor,\lceil n/2\rceil\}$ by deleting an arbitrary arc, or $D=H(V_1,V_2)$ with $\{|V_1|,|V_2|\}=\left\{n/2-1,n/2+1\right\}$ provided $n$ is even.
\end{corollary}
\begin{proof}
Define $V_1,V_2,V_1',V_1''$ as in Lemma \ref{le3}. If $V_1''$ is empty, then we have (\ref{eqh1}). By (\ref{eq11}) and (\ref{eq12}), equality in (\ref{eqh1}) holds if and only if $D$ is isomorphic to $B(V_1,V_2)$ with $\{|V_1|,|V_2|\}=\{\lfloor n/2\rfloor,\lceil n/2\rceil\}$.

If $|V_1''|\ge 1$, then we have (\ref{eqh22}). The condition (\ref{eq13}) leads to $|V_1''|= 1$ and
$$||V_1|-|V_2||=\left\{\begin{array}{ll}
1,& {\rm if}~  n ~{\rm  is~ odd};\\
0 {\rm~or~} 2, & {\rm if~} n ~{\rm is ~even}.
   \end{array}\right.$$
 Moreover, (\ref{eq11}) and (\ref{eq12}) implies that there is exactly one arc $e$ in $D[V_1]$ and all the other arcs in $D$ have  tails in $V_1$ and heads in $V_2$. If $n$ is odd, then $||V_1|-|V_2||=1$  implies that $D$ is isomorphic to a digraph obtained from $B(V_1,V_2)+e\in H(V_1,V_2)$ by deleting an arc in $B(V_1,V_2)$. If $n$ is even, then  $|V_1|=|V_2|=n/2$  implies that $D$ is isomorphic to the digraphs obtained from $B(V_1,V_2)+e$ by deleting an arc in $B(V_1,V_2)$, and $||V_1|-|V_2||=2$ implies that $D$ is isomorphic to $H(V_1,V_2)$ with $\{|V_1|,|V_2|\}=\{n/2-1,n/2+1\}$.

\end{proof}

\par

Now we are ready to present the proof of Theorem \ref{th1}.

{\bf Proof of Theorem \ref{th1}.}
Let $D\in T_{3,n}$. Then $D=T(V_1,V_2,V_3)+e$ with $|V_i|=\lfloor  {n}/{3}\rfloor$   or  $\lceil {n}/{3}\rceil ~{\rm for}~ i=1,2,3,$  and
$ e$  is an arc with both ends from    $V_2$. It is clear that all walks in $D$ have length $\le 3$. Moreover, for any walk $v_1v_2v_3v_4$ of length 3, we have $v_2v_3=e$. Therefore, there are no two distinct walks of length 3 with the same initial vertex and the same terminal vertex, i.e., $D$ is $F_3$-free. Hence,
\begin{equation}\label{eq4}
ex(n,F_3)\ge a(D)=\left\lfloor \frac{n^2}{3}\right\rfloor+1.
\end{equation}

Now let $D=(V,A)\in EX(n,F_3)$. Suppose $v_0$ is  a vertex with $d^+(v_0)=\Delta^+$. It follows from (\ref{eq4}) that
\begin{equation}\label{eq5}
\Delta^+ \ge \left\lfloor \frac{n}{3}\right\rfloor+1.
\end{equation}
Let $V_1= N^+(v_0)$ and  $V_2=V\setminus V_1$. We count $a(D)$ in the following way:
\begin{equation}\label{eq1}
a(D)=a(V_1,V_1)+a(V_1,V_2)+a(V_2,V)=a(V_1,V_1)+\sum\limits_{u\in V_2}[d^+(u)+d^-_{V_1}(u)].
\end{equation}
For any $u\in V_2$, if $d^-_{V_1}(u)\ge 2$, then $d^+(u)=0$  as $D$ is $F_3$-free. Therefore,
\begin{equation}\label{eq2}
d^+(u)+d^-_{V_1}(u)\le \Delta^++1
\end{equation}
 for all $u\in V_2$.
Moreover, equality in (\ref{eq2}) holds only if $d^+(u)=\Delta^+$ and $d^-_{V_1}(u)=1$. Now we prove the following claim.\\

\par
{\bf Claim 1.} $d^+(u)+d^-_{V_1}(u)\le\Delta^+$ for all $u\in V_2$.

{\bf Proof of Claim 1.} We first assert that there are at most two vertices in $V_2$ satisfying the equality in (\ref{eq2}). Suppose otherwise there are three vertices $u_1,u_2,u_3\in V_2$ satisfying the equality in (\ref{eq2}). Then we have  $$d^-_{V_1}(u_1)=d^-_{V_1}(u_2)=d^-_{V_1}(u_3)=1$$ and $$d^+(u_1)=d^+(u_2)=d^+(u_3)=\Delta^+.$$
Hence, there are $u_1',u_2',u_3'\in V_1$ such that $\{u_1'u_1,u_2'u_2,u_3'u_3\}\subseteq A$. By (\ref{eq5}), there are two vertices in $\{u_1,u_2,u_3\}$ sharing a common successor. Without loss of generality, we assume $u_1\rightarrow u$ and $u_2\rightarrow u$. Then $D$ contains the walks $v_0\rightarrow u_1'\rightarrow u_1\rightarrow u$ and $v_0\rightarrow u_2'\rightarrow u_2\rightarrow u$, a contradiction.

Since $D$ is $F_3$-free and $V_1$ is the out-neighbour set of $v_0$, $D[V_1]$  does not contain  two distinct walks of length 2 with a common terminal vertex. Applying Lemma \ref{le3} on $D[V_1]$, we have
$$a(D[V_1])=a(V_1,V_1)\le \left\lfloor\frac{(\Delta^+)^2}{4}\right\rfloor+1.$$
Combining this with (\ref{eq4}) and (\ref{eq1}), we obtain  $$\left\lfloor\frac{(\Delta^+)^2}{4}\right\rfloor+\Delta^+(n-\Delta^+)+3\ge a(D)\ge \left\lfloor\frac{n^2}{3}\right\rfloor+1,$$
which leads to
$$\frac{(\Delta^+)^2}{4}+\Delta^+(n-\Delta^+)+3\ge \frac{n^2}{3}.$$
It follows that
\begin{equation}\label{eq6}
\left\lceil\frac{2n}{3}\right\rceil-2\le \Delta^+\le\left\lfloor\frac{2n}{3}\right\rfloor+2.
\end{equation}

Suppose that there are $u_1,u_2\in V_2$ such that $$d^+(u_1)+d^-_{V_1}(u_1)=d^+(u_2)+d^-_{V_1}(u_2)=\Delta^++1.$$ Then $d^+(u_1)=d^+(u_2)=\Delta^+$ and $d^-_{V_1}(u_1)=d^-_{V_1}(u_2)=1$, which implies that there are $u_1',u_2'\in V_1$ such that $\{u_1'u_1,u_2'u_2\}\subseteq A$. By (\ref{eq6}), $u_1$ and $u_2$ share a common successor  $u$. Therefore,  $D$ contains the walks $v_0\rightarrow u_1'\rightarrow u_1\rightarrow u$ and $v_0\rightarrow u_2'\rightarrow u_2\rightarrow u$, a contradiction.

Note that $\Delta^+(n-\Delta^+)+\left\lfloor(\Delta^+)^2/4\right\rfloor$ is the size of a complete 3-partite (undirected) graph on $n$ vertices and $\left\lfloor n^2/3\right\rfloor$ is the size of a balanced 3-partite complete (undirected) graph on $n$ vertices. We always have
\begin{equation}\label{equ1}
\left\lfloor n^2/3\right\rfloor\ge \Delta^+(n-\Delta^+)+\left\lfloor(\Delta^+)^2/4\right\rfloor.
\end{equation}

Suppose there is only one vertex $u\in V_2$ satisfying the equality in (\ref{eq2}). Then we get
\begin{equation}\label{eq3}
d^+(u)=\Delta^+.
\end{equation} Moreover, $u$ has exactly one predecessor $w$ in $V_1$.
By (\ref{eq4}), (\ref{eq1}), (\ref{eq2}) and (\ref{equ1}), we get
\begin{eqnarray*}
a(V_1,V_1)&=&a(D)-\sum\limits_{u\in V_2}[d^+(u)+d^-_{V_1}(u)]\\
&\ge& \left\lfloor\frac{n^2}{3}\right\rfloor-\Delta^+(n-\Delta^+)\\
&\ge& \left\lfloor\frac{(\Delta^+)^2}{4}\right\rfloor.
\end{eqnarray*}
Therefore, we have either
\begin{equation}\label{eqh2}
a(V_1,V_1)=\left\lfloor\frac{(\Delta^+)^2}{4}\right\rfloor+1
\end{equation}
or
\begin{equation}\label{eqh3}
a(V_1,V_1)=\left\lfloor\frac{(\Delta^+)^2}{4}\right\rfloor.
\end{equation}
If (\ref{eqh2}) holds, then  $D[V_1]$ isomorphic to $H(U_1,U_2)$ with $\{|U_1|,|U_2|\}=\{\lfloor \Delta^+/2\rfloor,\lceil \Delta^+/2\rceil\}$. If (\ref{eqh3}) holds,  applying Corollary \ref{cor1} on $D[V_1]$, $D[V_1]$ is isomorphic to a digraph obtained from $H(U_1,U_2)$ with $\{|U_1|,|U_2|\}=\{\lfloor \Delta^+/2\rfloor,\lceil \Delta^+/2\rceil\}$ by deleting an arbitrary arc, or $D=H(U_1,U_2)$ with $\{|U_1|,|U_2|\}=\left\{\Delta^+/2-1,\Delta^+/2+1\right\}$ provided $\Delta^+$ is even.
In all the above cases, there is a subset of  $V_1$ with cardinality $\lceil \Delta^+/2\rceil-1$, say $V_3$, in which any pair of vertices has a common successor in $V_1$.

Since $D$ is $F_3$-free, $u$ has at most one successor in $V_3$.  Now $n\ge 16$, $d^+(u)=\Delta^+$ and (\ref{eq6}) enforce that $u$ has at least three successors $u_1,u_2,u_3$ in $V_2$.
On the other hand, we assert that there is at most one vertex $v\in V_2$ such that $$d^+(v)+d^-_{V_1}(v)\le \Delta^+-1.$$ Otherwise
we have
\begin{eqnarray*}
a(D)&=&a(V_1,V_1)+\sum\limits_{v\in V_2}[d^+(v)+d^-_{V_1}(v)]\\
&\le&\left(\left\lfloor \frac{(\Delta^+)^2}{4}\right\rfloor+1\right)+(n-\Delta^+)\Delta^+-1\\
&<&\left\lfloor \frac{n^2}{3}\right\rfloor+1,
\end{eqnarray*}
a contradiction.  Hence, two of  $u_1,u_2,u_3$, say $u_1$ and $u_2$, satisfy $d^+(u_i)+d^-_{V_1}(u_i)\ge \Delta^+-1$. It follows that  $d^+(u_i)\ge \Delta^+-2$ for $i=1,2$. Since  $u_1$ and $u_2$ have at most one successor in $V_3$,  they have a common successor $x$. Therefore, $D$ contains the walks $w\rightarrow u\rightarrow u_1\rightarrow x$ and $w\rightarrow u\rightarrow u_2\rightarrow x$, a contradiction.

 Therefore, there is no vertex $u$ in $V_2$ such that the equality in (\ref{equ1}) holds.
This completes the proof of Claim 1.
\qed\\

 \par

By Claim 1, (\ref{eq1}) and (\ref{equ1}), we get
\begin{eqnarray}
a(D)&=& a(V_1,V_1)+ \sum\limits_{u\in V_2}[d^+(u)+d^-_{V_1}(u)]\nonumber\\
&\le&\left\lfloor \frac{(\Delta^+)^2}{4}\right\rfloor+1+(n-\Delta^+)\Delta^+\label{equ2}\\
&\le&\left\lfloor \frac{n^2}{3}\right\rfloor+1\label{equ3}.
\end{eqnarray}
Combining this with (\ref{eq4}), we have
\begin{equation}\label{eqh13}
ex(n,F_3)=a(D)=\left\lfloor \frac{n^2}{3} \right\rfloor+1.
\end{equation}

Now we characterize the structure of $D$. The equality (\ref{eqh13}) implies that (\ref{equ2}) and  (\ref{equ3}) are both equalities. The equality in (\ref{equ3}) leads to
\begin{equation}\label{eq18}
\Delta^+=\left\{\begin{array}{ll}
 2n/3,&\text{if } n\equiv 0 \text{ (mod 3)};\\
\lfloor2n/3\rfloor \text{ or }\lfloor2n/3\rfloor+1,&\text{otherwise}.\\
\end{array} \right.
\end{equation}
The equality in (\ref{equ2}) implies
\begin{equation*}\label{eq9}
a(V_1,V_1)=\left\lfloor \frac{(\Delta^+)^2}{4}\right\rfloor+1
\end{equation*}
and
\begin{equation*}\label{eq10}
d^+(u)+d^-_{V_1}(u)=\Delta^+ \quad \text{for all}\quad u\in V_2.
\end{equation*}
Applying Lemma \ref{le3} on $D[V_1]$, $D[V_1]$ is isomorphic to $H(V_3,V_4)$ with $\{|V_3|,|V_4|\}=\{\lfloor\Delta^+/2\rfloor,$ $\lceil\Delta^+/2\rceil\}$.

Next we prove
\begin{equation}\label{eqh5}
a(V_1,V_2)=0.
 \end{equation}
 If there is a vertex $u\in V_2$ such that $d^-_{V_1}(u)\ge 2$, then $u$ has no successor, which implies  $d^-_{V_1}(u)=\Delta^+$. Therefore, all vertices in $V_4$ are predecessors of $u$. Choose any $u_1,u_2\in V_3$ and $u_3\in V_4$. We find two walks  $v_0\rightarrow u_1\rightarrow u_3\rightarrow u$ and $v_0\rightarrow u_2\rightarrow u_3\rightarrow u$ in $D$, a contradiction. Therefore, we have
\begin{equation}\label{eqh4}
d^-_{V_1}(u)\le 1 \quad\text{and}\quad d^+(u)\ge\Delta^+-1\quad\text{for all}\quad u\in V_2.
\end{equation}
Suppose  $d^-_{V_1}(u)=1$ with $w$ being the predecessor of $u$ in $V_1$. Then $d^+(u)=\Delta^+-1$. We assert that $u$ has at most one successor in $V_3$. Otherwise if $u$ has two successors $u_1,u_2$ in $V_3$, then $D$ contains two walks $w\rightarrow u\rightarrow u_1\rightarrow u_3$ and $w\rightarrow u\rightarrow u_2\rightarrow u_3$  for any $u_3\in V_4$, a contradiction.  By (\ref{eqh4}),   $u$ has at least two successors in $V_2$, say $u_1$ and $u_2$, which have a common successor $u_3$.
  Then $D$ contains two walks $w\rightarrow u\rightarrow u_1\rightarrow u_3$ and $w\rightarrow u\rightarrow u_2\rightarrow u_3$, a contradiction. Therefore, we have
  \begin{equation} \label{eq17}
d^-_{V_1}(u)=0 \quad\text{and}\quad d^+(u)=\Delta^+ \quad\text{for all}\quad u\in V_2,
\end{equation}
and (\ref{eqh5}) holds.

Finally, we assert
\begin{equation}\label{eqh6}
a(V_2,V_2)=0.
 \end{equation}
 Otherwise suppose   $u_1u_2\in A(D[V_2])$. If $u_2$ has two successors in $V_2$, say $u_3$ and $u_4$, by (\ref{eq18}) and (\ref{eq17}), $u_3$ and $u_4$ share a common successor $u_5$. Then $D$ contains two walks $u_1\rightarrow u_2\rightarrow u_3\rightarrow u_5$ and $u_1\rightarrow u_2\rightarrow u_4\rightarrow u_5$, a contradiction. If $u_2$ has two successors $u_3,u_4\in V_3$, then those successors share a common successor $u_5$ in $V_4$, and $D$ also contains the above walks, a contradiction again. It follows that
$$d^+(u_2)\le |V_2|+2<\Delta^+,$$ which contradicts (\ref{eq17}).
Hence, we have (\ref{eqh6}).

By (\ref{eqh5}), (\ref{eq17}) and (\ref{eqh6}), we see that $D[V_2]$ is an empty digraph and all vertices in $V_1$ are successors of $u$ for all $u\in V_2$. Therefore, $D=T(V_2,V_3,V_4)+e\in T_{3,n}$, where both ends of the arc $e$ are from $V_3$.
 This completes the proof.
\qed

\section*{Acknowledgement}
This work was supported by Science and Technology Foundation of Shenzhen City
(No. JCYJ20190808174211224)

\end{document}